\documentclass[12pt]{article}
\usepackage{amsmath,amssymb,amscd,amsthm}
\usepackage{latexsym}

\newtheorem{theorem}{Theorem}
\newtheorem{lemma}{Lemma}
\newtheorem{proposition}{Proposition}
\newtheorem{remark}{Remark}

\newcommand{\vol}[1]{\left| #1 \right|}
\newcommand{\abs}[1]{\left| #1 \right|}

\newcommand{\sprod}[2]{\left\langle #1, #2 \right\rangle}
\renewcommand{\[}{\left[}
\renewcommand{\(}{\left(}
\renewcommand{\]}{\right]}
\renewcommand{\)}{\right)}
\newcommand{\conv}{\mathop{\textup{conv}}}
\renewcommand{\le}{\leqslant}
\renewcommand{\ge}{\geqslant}
\newcommand{\lng}{\langle}
\newcommand{\rng}{\rangle}

\newcommand{\al}{\alpha}

\newcommand{\de}{\delta}
\newcommand{\De}{\Delta}

\newcommand{\ka}{\kappa}
\newcommand{\si}{\sigma}
\newcommand{\te}{\theta}

\newcommand{\ze}{\zeta}

\newcommand{\R}{\mathbb R}              % real numbers
\newcommand{\Sp}{{{\mathbb S}^{n-1}}}   % unit sphere
\newcommand{\VP}{\mathcal P}            % volume product

\newcommand{\F}{\mathcal F}
\newcommand{\FF}{\mathbb F}

\begin{document}

\title
{Local minimality of the volume-product at the simplex
\footnotetext{2000 Mathematics Subject Classification 52A40.}
}
\author{Jaegil Kim\thanks{Supported in part by U.S.~National Science Foundation grant DMS-0652684.}
 and Shlomo Reisner\thanks{Supported in part by the France-Israel Research Network Program
in Mathematics contract \#3-4301}
 }

\date{}
\maketitle

\begin{abstract}
It is proved that the simplex is a strict local minimum for the
volume product, $\mathcal{P}(K)=\min_{z\in {\rm int}(K)}
\vol{K}\vol{K^z}$, in the Banach-Mazur space of n-dimensional
(classes of ) convex bodies. Linear local stability in the
neighborhood of the simplex is proved as well. The proof consists of an extension to the non-symmetric setting of methods that were recently introduced by Nazarov, Petrov, Ryabogin
and Zvavitch, as well as proving results of independent interest, concerning stability of
square order of volumes of polars of non-symmetric convex bodies.
\end{abstract}

\section{Introduction and Preliminaries}
A \emph{body} is a compact set which is the closure of its
interior and, in particular, a \emph{convex body} in $\R^n$ is a
compact convex set with nonempty interior. If $K$ is a convex body
in $\R^n$ and $z$ is an interior point of $K$, then the
\emph{polar body} $K^z$ of $K$ with center of polarity $z$ is
defined by
$$
K^z = \{y\in\R^n : \sprod{y}{x-z}\le 1 \mbox{\ for all\ } x\in K\}
$$
where $\sprod{\cdot}{\cdot}$ is the canonical scalar product in
$\R^n$. In particular, if the center of polarity is taken to be the
origin, we denote by $K^\circ$ the polar body of $K$ and we clearly
have $K^z=(K-z)^\circ$.

If $A$ is a measurable set in $\R^n$ and $k$ is the minimal
dimension of a flat containing $A$, we denote by $\vol{A}$ the
$k$-dimensional volume (Lebesgue measure) of $A$. There should be no confusion
of the last notation with the notation for the Euclidean norm of a vector $x\in \R^n$
which is $|x|=\sqrt{\lng x,x\rng}$. A well known
result of Santal\'o \cite{S} states that in every convex body $K$ in
$\R^n$ there exists a unique point $s(K)$, called the
\emph{Santal\'o point} of $K$, such that
$$
\vol{K^{s(K)}} = \min_{z\in {\rm int}(K)} \vol{K^z}.
$$

The \emph{volume product} of $K$ is defined by
$$
\VP(K) = \inf \{\vol{K}\vol{K^z} : z\in {\rm int}(K)\}.
$$
A well known conjecture, called sometimes Mahler's conjecture (\cite{Ma1,Ma2}),
states that, for every convex body $K$ in $\R^n$,
\begin{equation}\label{eq:inverse_santalo}
\VP(K) \ge \VP(S)=\frac{(n+1)^{n+1}}{(n!)^2}
\end{equation}
where $S$ is an $n$-dimensional simplex. It is also conjectured that
equality in (\ref{eq:inverse_santalo}) is attained only if $K$ is a
simplex. The inequality (\ref{eq:inverse_santalo}) for $n=2$ was
proved by Mahler \cite{Ma1} with the case of equality proved by
Meyer \cite{Me91}. Other cases, like e.g. bodies of revolution, were
treated in \cite{MR98}. Several special cases in the centrally
symmetric setting can be found in \cite{SR,R86,GMR,Me86,R87}. Not
many special cases in which (\ref{eq:inverse_santalo}) is true seem
to be known, one such is proved in \cite{MR06}: all n dimensional
polytopes with at most n + 3 vertices (or facets). For more
information on Mahler's conjecture, see an expository article
\cite{Tao} by Tao.

The (non-exact) reverse Santal\'o inequality of Bourgain and Milman
\cite{BM} is
$$\VP(K) \ge c^n\VP(B^n_2)$$
where $c$ is a positive constant and $B^n_2$ is the Euclidean ball
(or any ellipsoid) in $\R^n$. Kuperberg \cite{Ku} reproved this
result with an improved constant. This should be compared with the
Blaschke-Santal\'o inequality
$$
\VP(K) \le \VP(B^n_2)
$$
with equality only for ellipsoids (\cite{S}, \cite{P}, see \cite{MP}
or also \cite{MR06} for a simple proof of both the inequality and
the case of equality)

The volume product is affinely invariant, that is,
$\VP(A(K))=\VP(K)$ for every affine isomorphism $A: \R^n \rightarrow
\R^n$. Thus, in order to deal with local behavior of the
volume product we need the following affine-invariant
(the {\em Banach-Mazur\/}) distance between
convex bodies:
$$
d_{BM}(K,L) = \inf \Big\{c : A(L)\subset B(K)\subset cA(L),
\text{for affine isomorphisms } A,B \text{ on}\,\R^n\Big\}
$$
If both $K$ and $L$ are symmetric convex bodies, this is just the
classical Banach-Mazur distance.

In a recent paper \cite{NPRZ}, the following result, connected to
the symmetric form of Mahler's conjecture, is proved:

\noindent\textbf{Theorem} \cite{NPRZ} {\it Let $K$ be an
origin-symmetric convex body in $\R^n$. Then
$$
\VP(K)\ge \VP(B_\infty^n),
$$
provided that $d_{BM}(K,B_\infty^n)\le 1+\de$, and $\de=\de(n)>0$ is
small enough (where $B_{\infty}^n$ is the $\ell_{\infty}^n$ unit ball).
 Moreover, the equality holds only if
$d_{BM}(K,B_\infty^n)=1$, i.e., if $K$ is a parallelopiped.}

In this paper we prove the analogous result for the $n$-dimensional simplex.

\begin{theorem}\label{th:main}
There exists $\de(n)>0$ such that the following holds:
Let $S$ be a simplex in $\R^n$ and $K$ a convex body in $\R^n$ with
$d_{BM}(K,S)=1+\de$ for  $0<\de<\de(n)$. Then
$$
\VP(K) \ge \VP(S) + C\de,
$$
where $C=C(n)$ is a positive constant.
\end{theorem}

There are some profound differences between the symmetric and
the non-symmetric cases. The most important one is, perhaps, the changed location
of the Santal\'o point when the body changes even slightly. Section~2 of this paper
deals with this change (it is shown that it obeys linear stability) and its implication
on the volume of the polar body (square-order stability). We believe that the results
of Section~2 have importance for their own sake.

Section~3 presents the necessary changes to the methods of \cite{NPRZ},
Among these we mention, in particular, the proof of Lemma~\ref{lem:deta_difference} here, which is the analogue of Section~5 of \cite{NPRZ}. This Lemma required a new proof that will work also in the non-symmetric setting. We provide here a coordinate-free proof that has a potential of being useful in other settings as well.

>From now on we fix a regular simplex $\De_n$ to be the convex hull of
$n+1$ vertices $v_0,\ldots,v_n$ where $v_0,\ldots,v_n$ are points on
$\Sp$ satisfying
$$
\sprod{v_i}{v_j} =\left\{
  \begin{array}{ll}
    1, & \hbox{if $i=j$} \\
    -\frac{1}{n}, & \hbox{otherwise.}
  \end{array}
\right.
$$
Note that $\De_n^{\circ}=-n\De_n$.
Most of the constants throughout the proofs depend on the dimension
n. They do not depend on the body K. We use the same letter (usually
$C$, $c$ etc.) to denote different constants in different paragraphs
or even in different lines.

\noindent{\bf Acknowledgment}. It is with great pleasure that we
thank Karoly Boroczky, Mathieu Meyer, Dmitry Ryabogin and Artem Zvavitch for very
helpful advice during the preparation of this paper.

\section{Continuity of the Santal\'o map}

The following volume formula is known (using polar coordinates).
For every interior point $z$ of $K$,
$$
\vol{K^z} = \frac{1}{n}\int_\Sp
(h_K(\te)-\sprod{z}{\te})^{-n}d\si(\te).
$$
where $\si$ is the spherical Lebesgue measure and $h_K$ is the
support function of $K$. By the minimum property of $|K^z|$ at the
Santal\'o point $s(K)$, it turns out that $z=s(K)$ is a
\emph{unique} point satisfying the condition (see \cite{Sc})
\begin{equation}\label{eq:santalo}
\int_\Sp (h_K(\te)-\sprod{z}{\te})^{-n-1}\te d\si(\te) =0.
\end{equation}
This is equivalent to the fact that the centroid of $K^z$ is the origin.

Denote by $\mathcal{K}^n$ the space of convex bodies in $\R^n$
endowed with the Hausdorff metric $d_H$. The space
$(\mathcal{K}^n,d_H)$ is isometrically embedded in the space
$C(\Sp)$ of continuous functions on the sphere $\Sp$ by the isometry
$K \mapsto h_K$, that is, $d_H(K,L)=\|h_K-h_L\|_\infty$ for every
$K,L$ in $\mathcal{K}^n$.

\begin{proposition}\label{prop:santalo_map}
The Santal\'o map $s:(\mathcal{K}^n,d_H)\rightarrow \R^n$ is
continuous. Furthermore, for every convex body $K_0$, there exist positive
constants $C=C(K_0)$ and $\de=\de(K_0)$ such that 
$$
d_H(K,K_0)\le \de \quad\Rightarrow\quad |s(K)-s(K_0)| \le
C\,d_H(K,K_0)\,.
$$

\end{proposition}

\begin{proof}
The continuity of $s(K)$ is proved using a standard argument (the
dominated convergence theorem and the uniqueness of $s(K)$ in
(\ref{eq:santalo})).

For the second part, fix a convex body $K_0$ and let $K$ be any
convex body which is close to $K_0$ in the Hausdorff metric. Since
$s(K_0)$ is in the interior of $K_0$, there is a $r_0>0$ such that
the ball $B(s(K_0),r_0)$ with center $s(K_0)$ and radius $r_0$ is
contained in $K_0$. Then, since
$$
K_0^{s(K_0)} = \( K_0 -s(K_0)\)^\circ \subset B(0,r_0)^\circ =
B(0,1/r_0),
$$
we have, for every $\te\in \Sp$,
$$
h_{K_0}(\te) - \sprod{s(K_0)}{\te} = \|\te\|_{K_0^{s(K_0)}} \ge r_0.
$$
Define three functions $a,x,y$ on $\Sp$ by
\begin{eqnarray*}
a(\te) &=& h_{K_0}(\te) - \sprod{s(K_0)}{\te}, \\
x(\te) &=& \sprod{s(K)}{\te} - \sprod{s(K_0)}{\te} \qquad\text{and}\\
y(\te) &=& h_K(\te) - h_{K_0}(\te).
\end{eqnarray*}
Note that, for every $\te \in \Sp$,
$$
|a(\te)| \ge r_0,\quad |x(\te)| \le |s(K)-s(K_0)|, \quad |y(\te)|
\le d_H(K,K_0)
$$
and $h_K(\te)-\sprod{s(K)}{\te} = a(\te) - x(\te) + y(\te)$. By the
Talyor formula, we can write as
\begin{eqnarray*}
(a-x+y)^{-n-1} &=& a^{-n-1}\Big[1-\frac{x-y}{a}\Big]^{-n-1}\\
&=& a^{-n-1}\Big[1+(n+1)\frac{x-y}{a} + f\(\frac{x-y}{a}\)\Big]
\end{eqnarray*}
where $f(t):=(1-t)^{-n-1}-1-(n+1)t$ is $O(t^2)$ for small $t$.
Then (\ref{eq:santalo}) implies
\begin{eqnarray*}
0 &=& \int_\Sp (h_K(\te)-\sprod{s(K)}{\te})^{-n-1}\te d\si(\te)
= \int_\Sp (a-x+y)^{-n-1}\te d\si(\te)\\
&=& \int_\Sp \Big( a^{-n-1} +\frac{(n+1)(x-y)}{a^{n+2}} +
\frac{f\(\frac{x-y}{a}\)}{a^{n+1}} \Big)
\,\te d\si(\te) \\
&=:& \quad (1) + (n+1)\Big[(2) - (3)\Big] + (4)
\end{eqnarray*}
where $(1),(2),(3)$ and $(4)$ are:

\begin{enumerate}
\item[(1)]
\begin{eqnarray*}
\int_\Sp a(\te)^{-n-1}\,\te d\si(\te) \quad=\, \int_\Sp
(h_{K_0}(\te)-\sprod{s(K_0)}{\te})^{-n-1}\te d\si(\te) \,=\,0.
\end{eqnarray*}

\item[(2)]
\begin{eqnarray*}
\abs{\int_\Sp \frac{x(\te)}{a(\te)^{n+2}}\,\te d\si(\te)} &=&
\abs{\int_\Sp \frac{\sprod{s(K)-s(K_0)}{\te}}
{\[h_{K_0}(\te)-\sprod{s(K_0)}{\te}\]^{n+2}}\,\te d\si(\te)} \\
&\ge& \sprod{\ze}{\int_\Sp \frac{\sprod{s(K)-s(K_0)}{\te}}
{\[h_{K_0}(\te)-\sprod{s(K_0)}{\te}\]^{n+2}}\,\te d\si(\te)}\\
&=& |s(K)-s(K_0)| \int_\Sp \frac{|\sprod{\ze}{\te}|^2 d\si(\te)}
{\[h_{K_0}(\te)-\sprod{s(K_0)}{\te}\]^{n+2}}\\
&\ge& |s(K)-s(K_0)| \int_\Sp \frac{|\sprod{\ze}{\te}|^2 d\si(\te)}
{|{\rm diam}(K_0)|^{n+2}}\\
&=& C_1 |s(K)-s(K_0)|,
\end{eqnarray*}
where $\ze = \frac{s(K)-s(K_0)}{|s(K)-s(K_0)|} \in\Sp$ and $C_1 =
\vol{B_2^n} |{\rm diam}(K_0)|^{-n-2}$.

\item[(3)]
$$
\abs{\int_\Sp \frac{y(\te)}{a(\te)^{n+2}}\,\te d\si(\te)} \,\,\le\,
\int_\Sp \frac{|y(\te)|}{|a(\te)|^{n+2}}\,|\te| d\si(\te) \,\le\,
\frac{1}{r_0^{n+2}}\, d_H(K,K_0).
$$

\item[(4)]
\begin{eqnarray*}
\abs{\int_\Sp \frac{f\(\frac{x-y}{a}\)}{a^{n+1}} \,\te\, d\si(\te)}
&\le& \int_\Sp \frac{c \abs{\frac{x(\te)-y(\te)}{a(\te)}}^2}{|a(\te)|^{n+1}} |\te| d\si(\te)\\
&\le& 2c \int_\Sp \frac{|x(\te)|^2 + |y(\te)|^2}{|a(\te)|^{n+3}} d\si(\te)\\
&\le& 2c \int_\Sp \frac{|s(K)-s(K_0)|^2 + d_H(K,K_0)^2}{r_0^{n+3}} d\si(\te)\\
&=& C_2\Big(|s(K)-s(K_0)|^2 + d_H(K,K_0)^2\Big),
\end{eqnarray*}
where $C$ is an absolute constant such that $|f(t)|\le C|t|^2$ for
$t$ near $0$ and $C_2=2cn\vol{B_2^n}r_0^{-n-3}$.
\end{enumerate}

Finally we have
\begin{eqnarray*}
&& C_1|s(K)-s(K_0)| \quad\le\quad |(2)| \quad=\quad \abs{(3)-\frac{1}{n+1}\Big((1)+(4)\Big)}\\
&& \quad\le\quad \frac{1}{r_0^{n+2}}\, d_H(K,K_0) +
\frac{C_2}{n+1}\Big(|s(K)-s(K_0)|^2+d_H(K,K_0)^2\Big).
\end{eqnarray*}
By continuity of $s(K)$ (and, in fact, local uniform continuity at $K_0$), $|s(K)-s(K_0)|\rightarrow 0$
whenever $d_H(K,K_0)\rightarrow 0$. Thus the two quadratic terms in the
inequality above can be ignored whenever $\de$ is small enough.
Therefore $$|s(K)-s(K_0)| \le C\, d_H(K,K_0)$$ where $C$ is a
constant
greater than $|{\rm diam}(K_0)|^{n+2}\vol{B_2^n}^{-1} r_0^{-n-2}$. \\
\end{proof}

\begin{proposition}\label{prop:santalo_vol}
Let $K$ be a convex body in $\R^n$. If $z\in \mbox{int}(K)$ is close enough to $s(K)$ then
$$|K^z|\leq|K^{s(K)}|\left(1+\frac{c}{r_0^2}|z-s(K)|^2\right)\,,$$
where $r_0>0$ is such that $B(s(K),r_0)\subset K$ and $c$ is a constant independent of $K$.
\end{proposition}

\begin{proof}
Assume that $0$ is the Santal\'o point of $K$. Then $K^{s(K)}=K^{\circ}$,
$$|K^{\circ}|=\int_{K^{\circ}}\,dy\mbox{\quad and\quad} |K^{z}|=
\int_{K^{\circ}}\frac{dy}{(1-\lng z,y\rng)^{n+1}}$$
(cf.\ e.g.\ Lemma~3 of \cite{MW}).
Note that $K^{\circ}\subset B(0,r_0^{-1})$, Hence $|\sprod{z}{y}|\leq \frac{|z|}{r_0}$ for $y\in K^{\circ}$.
We represent
$(1-t)^{-(n+1)}$ as $1+(n+1)t+g(t)$. We have
$$\int_{K^{\circ}}\lng z,y\rng\,dy=0$$
because $0$ is the Santal\'o point of $K$. Thus
$$|K^{z}|=|K^{\circ}|+\int_{K^{\circ}}g(\lng z,y\rng)\,dy$$
and we get
$$|K^{z}|\leq |K^{\circ}|\left(1+\sum_{j=2}^{\infty}
\frac{(n+1)(n+2)\ldots(n+j)}{j!}\left(\frac{|z|}{r_0}\right)^j\right)\,.$$

We finally have
$$
|K^{z}|\leq
|K^{\circ}|\left(1+c\left(\frac{|z|}{r_0}\right)^2\right)
$$
if $|z|\leq \frac{r_0}{2}$ (say).\\
\end{proof}

\begin{remark}
Under the assumptions of Proposition~\ref{prop:santalo_vol}, if
$|z-s(K)|<r_0$ and $K$ contains a ball $B(z,2r_0)$ then it
contains $B(s(K),r_0)$ this will be used later in the application
of Proposition~\ref{prop:santalo_vol}.
\end{remark}

\section{Construction of auxiliary Polytopes}

In this section we prove an analogue of \cite{NPRZ} for the
$n$-dimensional simplex. Thus most of the ideas and tools that are used in the proofs
in this section
are basically adaptations of those from \cite{NPRZ}. Lemma~\ref{lem:deta_difference},
which replaces Section~4 of \cite{NPRZ}, had to be worked out anew and to be put
on a less coordinate dependent basis.

Let $F$ be a $k$-dimensional face of $\De_n$ for $0\le k< n$ and
denote by $c_F$ its centroid. Consider the affine subspace
$H_F=c_F+F^\bot$ where $F^\bot=\{y\in\R^n : \sprod{x}{y}=0 \,\forall
x\in F\}$. Take a $t>0$ such that $tH_F$ is tangent to $K$. In case
that $(1-\de)\De_n \subset K\subset \De_n$, it should be $1-\de\le
t\le 1$.  Let $x_F$ be such a tangent point, that is, $x_F \in
tH_F\cap\partial K$ and put $y_F = tc_F$. Denote the dual face of
$F$ by $F^*=\{y\in\De_n^\circ : \sprod{x}{y}=1 \,\forall x\in F\}$.
By the same way as above, we have points $x_F^*$ and $y_F^*$ by
replacing $F$, $K$ and $\De_n$ by $F^*$, $K^\circ$ and
$\De_n^\circ$, respectively. These four points $x_F$, $y_F$, $x_F^*$
and $y_F^*$ have the following properties.

\begin{lemma}\label{lem:construction}
Let $F$ be a face of $\De_n$. Suppose that $(1-\de)\De_n \subset
K\subset \De_n$. Then
\begin{enumerate}
   \item $\sprod{x_F}{x_F^*} =1= \sprod{y_F}{y_F^*}$
   \item $\sprod{x_F-y_F}{c_F} =0= \sprod{x_F^*-y_F^*}{c_F}$
   \item $|x_F-y_F| < 2\de$ and $|x_F^*-y_F^*| < 2n\de$.
 \end{enumerate}
\end{lemma}
\begin{proof}
1. Let $t,s>0$ be such that $x_F\in tH_F\cap\partial K$ and
$x_F^*\in sH_{F^*}\cap\partial K^\circ$. Then we can check easily
that
\begin{eqnarray*}
H_{F^*} &:=& c_{F^*}+(F^*)^\bot \\
&=& \{z\in\R^n : \sprod{z}{h}=1 \,\forall h\in H_F\}.
\end{eqnarray*}
Consider a hyperplane $G$ containing $tH_F$ which is tangent to $K$
at $x_F$ and let $\al$ be the dual point of $G$, that is,
$\sprod{\al}{z}=1$ for every $z\in G$. So $\sprod{\al}{th}=1$ for
all $h\in H_F$ which implies $t\al\in H_{F^*}$. Since
$\al\in\partial K^\circ$ by construction of $G$, we get $\al\in
\frac{1}{t}H_{F^*}\cap\partial K^\circ$ which implies $s=1/t$. After
all, $\sprod{x_F}{x_F^*}=1$ since $x_F\in tH_F$ and $x_F^*\in
\frac{1}{t}H_{F^*}$, and
$\sprod{y_F}{y_F^*}=\sprod{tc_F}{\frac{1}{t}c_{F^*}}=1$.

2. Note that $\sprod{x_F-y_F}{c_F} = \sprod{x_F}{c_F} -
\sprod{y_F}{c_F} =t-t=0$. Similarly, we have
$\sprod{x_F^*-y_F^*}{c_{F^*}} =0$. Thus $\sprod{x_F^*-y_F^*}{c_F}=0$
since $c_{F^*}=\frac{1}{|c_F|^2}c_F$.

3. Write $F=\conv(v_0,\ldots,v_k)$. Then $F^\bot$ is in the linear
span of $v_{k+1},\ldots,v_n$ and hence $tH_F = tc_F + F^\bot$ should
be in the linear span of $v_{k+1},\ldots,v_n$ and $c_F$. Thus
\begin{eqnarray*}
tH_F\cap \De_n &=& (tc_F + F^\bot)\cap \conv(v_0,v_1,\ldots,v_n)\\
&=& (tc_F + F^\bot)\cap \conv(c_F,v_{k+1},\ldots,v_n)\\
&\subset& tc_F + (1-t)\conv(v_{k+1},\ldots,v_n).
\end{eqnarray*}
Therefore
\begin{eqnarray*}
|x_F-y_F| &\le& {\rm diam}(tH_F\cap\De_n) \le (1-t){\rm
diam}\big(\conv(v_{k+1},\ldots,v_n)\big)\\
&\le& {\rm diam}(\De_n)\, \de.
\end{eqnarray*}
Similarly, we get $|x_F^*-y_F^*|\le{\rm diam}(\De_n^\circ)\, \de$.
\end{proof}

Let $\F$ be the set of all faces of $\De_n$. A family $\FF$ of $n$
faces $F_0,\ldots,F_{n-1}$ in $\F$ is called a \emph{flag} over $\F$
if each $F_k$ is a $k$-dimensional face in $\F$ and $F_0 \subset F_1
\subset \cdots \subset F_{n-1}$.

For each face $F\in\F$, we constructed four points $x_F$, $x_F^*$,
$y_F$ and $y_F^*$ in the previous paragraph. These points induce
the following four polytopes (in general, not convex):
\begin{eqnarray*}
P &=& \bigcup_\FF \, \conv\(0,x_{F_0},\ldots,x_{F_{n-1}}\),\quad
P' = \bigcup_\FF \, \conv\(0,x^*_{F_0},\ldots,x^*_{F_{n-1}}\),\\
Q &=& \bigcup_\FF \, \conv\(0,y_{F_0},\ldots,y_{F_{n-1}}\),\quad Q'
= \bigcup_\FF \, \conv\(0,y^*_{F_0},\ldots,y^*_{F_{n-1}}\)
\end{eqnarray*}
where $\FF:=\{F_0,\ldots,F_{n-1}\}$ runs over all flags of $\F$.
Under the assumption $(1-\de)\De_n \subset K\subset \De_n$, they
clearly satisfy $P\subset K$, $P'\subset K^\circ$, $(1-\de)\De_n
\subset Q\subset \De_n$ and $\De_n^\circ \subset Q' \subset
\frac{1}{1-\de}\De_n^\circ$.

\begin{lemma}\label{lem:polytope_Q}
$\vol{Q}\cdot\vol{Q'} \ge \vol{\De_n}\cdot\vol{\De_n^\circ}$
\end{lemma}

The proof is essentially the same as the proof of Lemma 7 of \cite{NPRZ}.

\begin{lemma}\label{lem:vol_difference}
Suppose that $(1-\de)\De_n \subset K\subset \De_n$. Then there exist
constants $C_1$ and $C_2$ such that $\big|\vol{P}-\vol{Q}\big|\le
C_1\de^2$ and $\big|\vol{P'}-\vol{Q'}\big|\le C_2\de^2$
\end{lemma}
\begin{proof}
We can check that Lemma 4 of \cite{NPRZ} is also true for the simplex $\De_n$.
This fact, together with Lemma 1 here and Lemma 5 of \cite{NPRZ} (taking $X_0=\{c_F\}$, $X_1=\{x_F\}$, $X_2=\{y_F\}$ and similarly for the starred points), completes the proof of the lemma.
\end{proof}

Suppose that all the centroids of facets of $\De_n$ belong to $K$.
Then, for every facet $F$ of $\De_n$,
$$
x_F=y_F=c_F \quad\text{and}\quad x_F^*=y_F^*=c_{F^*}.
$$
This is helpful in the proof of the following lemma.
\begin{lemma}\label{lem:deta_difference}
There exists $c'>0$ such that if $\de = \min \{ d>0 : (1-d)\De_n
\subset K \subset \De_n \}$ is small enough and if all the
centroids of facets of $\De_n$ belong to $K$, then $\vol{K} \ge
\vol{P} + c'\de$ or $\vol{K^\circ} \ge \vol{P'} + c'\de$.
\end{lemma}

\begin{proof}
We begin by proving that there exists a constant $c_1>0$ such that
$(1+c_1\de)P'\not\supset P^{\circ}$.\quad
Since $\de$ is the minimal number that satisfies $(1-\de)\Delta_n\subset K$, we can find a vertex
$v_j$ of $\Delta_n$, say $v_0$, such that $(1-\de)v_0\in\partial K$. Taking $F_0=\{v_0\}$
in Lemma~\ref{lem:construction}, we conclude the existence of
$$x_0=tv_0+h\in \partial K,\mbox{\ with\ } h\in v_0^{\perp},\quad |h|<C\de,\quad
1-\de\leq t\leq 1,$$
and
$$x_0^*=sv_0\in \partial K^{\circ},\mbox{\ with\ } 1\leq s\leq \frac{1}{1-\de},$$
such that $\langle x_0,x_0^*\rangle=ts=1$.

Let $z^*\in \partial K^{\circ}$ be such that $\lng z^*,(1-\de)v_0\rng=1$. Then
$H=\{x\,;\,\lng z^*,x \rng=1\}$ is a support hyperplane of $K$ at $(1-\de)v_0$. Thus
$\lng x_0,z^* \rng \leq 1$. Since $H$ is also a support hyperplane of $(1-\de)\Delta_n$ at
$(1-\de)v_0$, it follows that $x_0$ lies below the one-sided cone $\cal C$ with
vertex $(1-\de)v_0$,
which is the complimentary half of the cone with the same vertex, spanned by
$(1-\de)\Delta_n$. Take a typical facet $G$ of the cone $\cal C$. Say
$G\subset\{x\,;\, \lng x,-nv_1 \rng=
1-\de\}$. The highest point (with respect to the direction $v_0$) of $G\cap \Delta_n$ is the
intersection of $G$ with the line segment
$[v_0,v_1]$. A simple calculation shows that this is the point $\beta v_0+(1-\beta)v_1$
with $\beta=1-\frac{\de}{n+1}$. The height of this point is found by computing its projection
on the altitude $[v_0,-\frac{v_0}{n}]$ of $\Delta_n$. This is
$\beta v_0+(1-\beta )(-\frac{v_0}{n})=(1-\frac{\de}{n})v_0$. we conclude that
$$t\leq 1-\frac{\de}{n}\mbox{\ and\ }s\geq \frac{1}{1-\frac{\de}{n}}\,.$$
Thus
$$x_0^*=sv_0\mbox{\ with\ }\frac{1}{1-\frac{\de}{n}}\leq s\leq
\frac{1}{1-\de}\,.$$

We look now at the vector $h\in v_0^{\perp}$ that was found above
(with $x_0=tv_0+h$). There exists one of the vectors $-nv_j-v_0$,
$j=1,\ldots,n$, which are vertices in $v_0^{\perp}$ of a regular
simplex with center $0$, such that $\lng -nv_j-v_0,h \rng\leq 0$. We
may assume that this $j$ is $1$ and denote $v=\gamma(-nv_1-v_0)\in
v_0^{\perp}$ for some $\gamma>0$ whose size will be determined
later. Note that $|v|=\gamma\sqrt{n^2-1}$ and that $\lng v,h
\rng\leq 0$.

Define $\tilde x=x_0^*+v$. We claim that if $\gamma$ is chosen correctly then $\tilde x\in P^{\circ}$
and, for some $c_1>0$, $\tilde x\not\in(1+c_1)P'$. To verify that $\tilde x\in P^{\circ}$ we have to
check that $\lng \tilde x,x_F \rng\leq 1$ for all the vertices $x_F$ of $P$.
These vertices are of the form $x_F=\frac{1}{k}\sum_{i=1}^kv_{j_i}+g$, $1\leq k\leq n$, $|g|<C\de$
and $g=0$ if $k=n$.
\begin{itemize}
\item[1)] Let $x_{F_0}=x_0=tv_0+h$. Then $\lng \tilde x,x_{F_0} \rng=\lng x_0^*,x_0 \rng +\lng v,x_0 \rng
=1+\lng v,h \rng\leq 1$.
\item[2)] Let $x_F=\frac{1}{k}\sum_{i=1}^kv_{j_i}+g$, $1\leq k\leq n$, $|g|<C\de$. Assume first
that the index $0$ is not among the $j_i$-s. Say
$x_F=\frac{1}{k}\sum_{j=1}^kv_{j}+g$ (it is true that $v_1$ plays a somewhat different role
than the other indices $j\geq 1$, but the result of the coming evaluation comes up to be the same). Then
$$\lng \tilde x,x_F \rng=\frac{1}{k}\sum_{j=1}^ks\lng v_0,v_j\rng+\lng v,\frac{1}{k}\sum_{j=1}^kv_{j}\rng+
\lng sv_0,g\rng+\lng v,g\rng\,.$$
We have $|\frac{1}{k}\sum_{j=1}^kv_{j}|=\sqrt{\frac{n-k+1}{nk}}$ thus
$$\lng \tilde x,x_F \rng\leq s\left(-\frac{1}{n}+C\de\right)+|v|\left(\sqrt{\frac{n-k+1}{nk}}+C\de\right)<1$$
for small $\de$, if $\gamma<\frac{c_2}{n}$ for an appropriate constant $c_2$.
\item[3)] Let $x_F$ be as in 2) above, now with $0$ among the $j_i$-s, say
$x_F=\frac{1}{k}\sum_{j=0}^{k-1}v_{j}+g$ (same remark about $v_1$). The calculation now gives
$$\lng \tilde x,x_F \rng\leq \frac{s}{k}-\frac{s(k-1)}{nk}+
|v||\frac{1}{k}\sum_{j=0}^{k-1}v_{j}|+(s+|v|)|g|\leq$$
$$s\left(\frac{1}{k}-\frac{k-1}{kn}+C\de\right)+|v|\left(\sqrt{\frac{n-k+1}{nk}}+C\de\right)<1$$
if $\de$ is small
and $\gamma\leq\frac{c_2}{n}$ (note that in this case $k\geq 2$).
\end{itemize}

We fix now the constant $\gamma$ that was introduced above to be precisely $\frac{c_2}{n}$
with the constant $c_2$ obtained above. Then $\tilde x\in P^{\circ}$ (provided that $\de$ is small
enough). As $\tilde x$ is a positive linear combination of $x_0^*=sv_0$ and
$-nv_1$, the line segment connecting the origin to $\tilde x$ must cross the edge
$[sv_0,-nv_1]$ of $P'$. Thus we look for $M>0$ and $0<\theta<1$ such that the equality
$$sv_0+v=M\left(\theta sv_0+(1-\theta)(-nv_1)\right)$$
will hold. Substituting $v=\gamma(-nv_1-v_0)$ we get $M=1+\gamma(1-\frac{1}{s})$.
As we had the evaluation $\frac{1}{s}=t\leq 1-\frac{\de}{n}$, we get
$$M\geq 1+\frac{\gamma}{n}\de=1+\frac{c_2}{n^2}\de\,.$$
That is, if $c_1<\frac{c_2}{n^2}$ then $\tilde x\not\in (1+c_1\de)P'$  and we get
$(1+c_1\de)P'\not\supset P^{\circ}$.\vspace{3mm}

Assume that $K\subset (1+\frac{c_1}{2}\de)\conv(P)$. Then
$(1-\frac{c_1}{2}\de)P^{\circ}\subset \frac{1}{1+\frac{c_1}{2}\de}P^{\circ}\subset K^{\circ}$. Let
$$\tilde{\tilde x}=(1-\frac{c_1}{2}\de)\tilde x\in (1-\frac{c_1}{2}\de)P^{\circ}\subset K^{\circ}\,.$$
By the preceding paragraph we have $\tilde{\tilde x}\not\in (1+c_1\de)(1-\frac{c_1}{2}\de)P'$.
As $(1+c_1\de)(1-\frac{c_1}{2}\de)>1+\frac{c_1}{4}\de$ if $\de<\frac{1}{2c_1}$, we conclude
that for $\de$ small enough, either $K\not\subset (1+\frac{c_1}{2}\de)\conv(P)$; in which case,
by Lemma~2 of \cite{NPRZ}, $|K|\geq |P|+c_3\de$; or there exists
$\tilde{\tilde x}\in K^{\circ}$, such that the line segment $[0,\tilde{\tilde x}]$ intersects
the edge $[x_0^*,-nv_1]$ of $P'$, but $\tilde{\tilde x}\not\in (1+\frac{c_1}{4}\de)P'$. That is,
$K^{\circ}\not\subset (1+c_4\de)P'$. Lemma~2 of \cite{NPRZ} completes now the proof. We remark,
that, since $P'$ is, in general, not convex the assumption of the uniform lower bound on the
$(n-1)$-dimensional volume of its facets should be verified using the $\de$-approximation.
\end{proof}

\begin{proposition}\label{prop:construction}
Let $K$ be a convex body in $\R^n$ which is close to $\De_n$ in the
sense that $\de = \min \{ d>0 : (1-d)\De_n \subset K \subset \De_n
\}$ is small enough. Suppose that all the centroids of facets of
$\De_n$ belong to $K$. Then we have
$$\vol{K}\vol{K^\circ} \ge \vol{\De_n}\vol{\De_n^\circ} +C\de.$$
\end{proposition}

\begin{proof}
We assume that $\vol{K} \ge \vol{P}+c\de$ by Lemma
\ref{lem:deta_difference}. Moreover, Lemma \ref{lem:vol_difference}
implies
\begin{eqnarray*}
\vol{K}\vol{K^\circ} &\ge& (\vol{P} + c\de) \vol{P'}\\
&\ge& (\vol{Q}-c_1\de^2+c\de) (\vol{Q'} - c_2\de^2) \\
&=& \vol{Q}\vol{Q'} + \vol{Q'}(c\de-c_1\de^2) - c_2\vol{Q}\de^2 -c_2\de^2(c\de-c_1\de^2) \\
&\ge& \vol{Q}\vol{Q'} + \vol{\De_n^\circ}(c\de-c_1\de^2) -
c_2\vol{\De_n}\de^2 -c_2\de^2(c\de-c_1\de^2)
\end{eqnarray*}
Since $\de$ is small enough, the above inequality implies that
$\vol{K}\vol{K^\circ} \ge \vol{Q}\vol{Q'} + C\de$ for a constant
$0<C< \vol{\De_n^\circ}c$. Finally, Lemma \ref{lem:polytope_Q}
completes the proof.

\end{proof}

\section{Proof of Theorem \ref{th:main}}

For the proof of main theorem, let us start with the following
lemma.
\begin{lemma}\label{lem:distance}
Let $L$ be a convex body in $\R^n$ containing the origin. Then, for
every convex body $K$ with $d_{BM}(K,L)<1+\de$, there are a constant
$C=C(L)$ and an affine isomorphism $A:\R^n\rightarrow\R^n$ such that
$$(1-C\de)L \subset A(K) \subset L.$$ In particular, if $L=\De_n$ and
$\de>0$ is small enough, then such $C$ and $A$ can be chosen to
satisfy that every centroid of facets of $L$ belongs to $A(K)$.
\end{lemma}

\begin{proof}
By definition, there are affine isomorphisms
$A,B:\R^n\rightarrow\R^n$ such that
$$
(1-\de)A(L) \subset B(K) \subset A(L).
$$
Clearly $A(L)$ should contain the origin $0$. Put $a=A^{-1}(0)$.
Then it is in $L$ and we can write $A(x) = T(x-a)$, $(x\in\R^n)$ for
some linear transformation $T$ on $\R^n$. Note that, for every point
$x$,
\begin{eqnarray*}
A^{-1}\Big((1-\de)A(x)\Big) &=& A^{-1}\Big((1-\de)T(x-a)\Big) =
A^{-1}T\Big((1-\de)x+a\de-a\Big) \\
&=& A^{-1}A\Big((1-\de)x+a\de\Big) = (1-\de)x +a\de,
\end{eqnarray*}
which implies $A^{-1}\Big((1-\de)A(L)\Big) = (1-\de)L + a\de$. Take
a constant $c>1$ such that $-L\subset cL$. We have
\begin{eqnarray*}
\Big(1-(1+c)\de\Big)L -a\de &\subset& \Big(1-(1+c)\de\Big)L -\de L  \\
&\subset& \Big(1-(1+c)\de\Big)L + c\de L \quad=\quad (1-\de)L
\end{eqnarray*}
The above two facts imply
\begin{eqnarray*}
\Big(1-(1+c)\de\Big)L &\subset& A^{-1}\Big((1-\de)A(L)\Big)\\
&\subset& A^{-1}B(K) \quad\subset\quad L.
\end{eqnarray*}

For the case $L=\De_n$, note that $-\De_n \subset n\De_n$. Thus we
have
$$
(1-(n+1)\de)\De_n \subset A^{-1}B(K) \subset \De_n.
$$
Let $S$ be a simplex of minimal volume containing $K_1:=A^{-1}B(K)$.
It is proved in \cite{Kl} that all the centroids of facets of $S$
belong to $K_1$. On the other hand, if $x\in S\setminus
(1+\ka)\Big((1-(n+1)\de)\De_n\Big)$, then by Lemma 2 of \cite{NPRZ},
\begin{eqnarray*}
\vol{S} &\ge& \vol{\Big(1-(n+1)\de\Big)\De_n} +
\frac{\ka\big(1-(n+1)\de\big)^n\frac{1}{n}\(\frac{n^2}{n+1}\vol{\De_n}\)}{n}
\\
&=& \vol{\De_n}\Big(1-(n+1)\de\Big)^n\(1+\frac{\ka}{n+1}\).
\end{eqnarray*}
If $\ka > \ka_0 = (n+1)\[(1-(n+1)\de)^{-n}-1\]$, the existence of
such $x$ implies $\vol{S} > \vol{\De_n}$ which is a contradiction.
Hence, for $\ka > \ka_0$,
\begin{eqnarray*}
(1-\ka)S &\subset& (1-\ka)\[(1+\ka)\big(1-(n+1)\de\big)\De_n\]\\
&\subset& \big(1-(n+1)\de\big)\De_n \\
&\subset& K_1 \quad\subset\quad S.
\end{eqnarray*}
Note also that there exists a unique affine isomorphism $A_1$
satisfying $S=A_1(\De_n)$. Applying the argument of the first part
again, we have
$$
\big(1-(n+1)\ka\big)\De_n \subset A_1^{-1}(K_1) \subset \De_n.
$$
where $\ka > (n+1)\Big((1-(n+1)\de)^{-n}-1\Big) \approx n(n+1)^2\de$
if $\de>0$ is small enough.

\end{proof}

\begin{proof}[\bf Proof of Theorem~\ref{th:main}]
Let $K$ be a convex body with $d_{BM}(K,S)=1+\de$ for sufficiently
small $\de>0$. By Lemma \ref{lem:distance}, (replacing $S$ by
$\De_n$), there is a constant $C=C(n)>0$ such that
$$
(1-C\de)\De_n \subset A(K) \subset \De_n.
$$
and all the centroids of facets of $\De_n$ belong to $A(K)$. Since the
volume product is invariant under affine isomorphisms of $\R^n$,
we may assume that all the centroids of facets of $\De_n$
belong to $K$. We may also ``include'' the constant $C$ into $\de$ and
assume that
$$
\de := \min \{ d>0 :(1-d)\De_n \subset K \subset \De_n \}\,.
$$
This implies that $d_H(K,\De_n) \le \de$. From Proposition
\ref{prop:santalo_map} we now conclude that $|s(K)|\le c_1\de$ for
some $c_1>0$. By Proposition~\ref{prop:santalo_vol}, and the
remark following it, we get the inequality
\begin{eqnarray*}
\vol{K^\circ} &\le& \vol{K^{s(K)}}(1+ c_2|s(K)|^2) \\
&\le& \vol{K^{s(K)}}(1 + c_1c_2 \de^2)
\end{eqnarray*}
for some constant $c_2>0$ (a-priory, $c_2$ would depend on the
radius of a ball centered at $s(K)$ and contained in $K$. The remark
following Proposition~\ref{prop:santalo_vol} allows us to use instead
a ball centered at $0$. The relation $(1-\de)\De_n\subset K$ then
allows us to use a ball contained in, say,  $\frac{1}{2}\De_n$ instead.
Thus $c_2$ may be considered as independent of $K$). Hence
$$
\vol{K^{s(K)}}\geq \vol{K^\circ}-c_1c_2\vol{K^{s(K)}}\de^2\geq
\vol{K^\circ}(1-c_1c_2\de^2),
$$
and the volume product of $K$ satisfies
\begin{eqnarray*}
\VP(K) = \vol{K}\vol{K^{s(K)}} &\ge& \vol{K}\vol{K^\circ}(1- c_1c_2
\de^2).
\end{eqnarray*}
Proposition \ref{prop:construction} implies that
$\vol{K}\vol{K^\circ} \ge \vol{\De_n}\vol{\De_n^\circ} + c\de$.
Finally, we have
\begin{eqnarray*}
\VP(K) &\ge& \Big(|\De_n||\De_n^\circ| + c\de\Big)(1-c_1c_2\de^2) \\
&\ge& \VP(S) + C\de.
\end{eqnarray*}
for sufficiently small $\de>0$ and a constant $C>0$.
\end{proof}

\small{
\noindent J. Kim: Department of Mathematics, Kent State
University, Kent, OH 44242, USA.\newline Email:  jkim@math.kent.edu

\noindent
S. Reisner: Department of Mathematics, University of Haifa, Haifa 31905, Israel.\newline
Email: reisner@math.haifa.ac.il}

\end{document}